\pdfoutput=1
\RequirePackage{fix-cm}
\documentclass[letterpaper,10pt,conference]{ieeeconf}
\IEEEoverridecommandlockouts
\usepackage{amsmath,amssymb,mathtools}
\usepackage{graphicx,cite}
\usepackage[hidelinks]{hyperref}
\usepackage{microtype}
\usepackage{stfloats}
\usepackage[checkfloat,nocheckfootnote]{flushend}
\usepackage{needspace}
\usepackage{placeins}
\newcommand{\R}{\mathbb R}
\newcommand{\E}{\mathbb E}
\newcommand{\Prob}{\mathbb P}
\newcommand{\btheta}{\boldsymbol{\theta}}
\newcommand{\thh}{\hat{\btheta}}
\newcommand{\Lstar}{L(\btheta^*)}
\newcommand{\yb}{\bar{y}}
\newcommand{\epsb}{\bar{\epsilon}}
\newcommand{\newth}[1]{\btheta_{\mathrm{new}}(#1)}
\newtheorem{theorem}{Theorem}
\makeatletter
\def\@begintheorem#1#2{\@IEEEtmpitemindent\itemindent\topsep\z@\rmfamily\trivlist
  \item[\hskip\labelsep{\itshape #1\ #2:}]\itemindent\@IEEEtmpitemindent}
\def\@opargbegintheorem#1#2#3{\@IEEEtmpitemindent\itemindent\topsep\z@\rmfamily\trivlist
  \item[\hskip\labelsep{\itshape #1\ #2\ (#3):}]\itemindent\@IEEEtmpitemindent}
\makeatother
\renewenvironment{proof}{\par\noindent\textit{Proof. }\ignorespaces}{\par}
\newenvironment{assumptions}{\begin{list}{}{\setlength{\leftmargin}{0pt}\setlength{\itemindent}{0pt}\setlength{\itemsep}{4pt}\setlength{\parsep}{0pt}}}{\end{list}}

\renewenvironment{abstract}{\par\normalfont\small\noindent\textbf{Abstract}---\,\ignorespaces}{\par}
\newcommand{\proofpart}[1]{\par\smallskip\noindent\textbf{#1}\par\nobreak\noindent\ignorespaces}
\graphicspath{{figures/}}
\hypersetup{pdftitle={Blind Random Search with Noisy Measurements: Almost Sure Convergence and Convergence Rates},pdfauthor={Zixian Zhou and James C. Spall}}

\title{\LARGE\bf Blind Random Search with Noisy Measurements:\\Almost Sure Convergence and Convergence Rates}
\author{Zixian Zhou\textsuperscript{1}\thanks{\textsuperscript{1}Zixian Zhou is with
 the Department of Applied Mathematics and Statistics, Johns Hopkins University,
 Baltimore, MD, USA, \href{mailto:zzhou147@jh.edu}{zzhou147@jh.edu}.} and
 James C. Spall\textsuperscript{2}\thanks{\textsuperscript{2}James C. Spall is with
 the Department of Applied Mathematics and Statistics, Johns Hopkins University,
 Baltimore, MD, USA and the Johns Hopkins University Applied Physics Laboratory,
 Laurel, MD, USA. Corresponding author:
 \href{mailto:james.spall@jhuapl.edu}{james.spall@jhuapl.edu}.}}
\begin{document}
\maketitle
\thispagestyle{empty}\pagestyle{empty}

\begin{abstract}
Blind random search is a simple global optimization method that only requires the value of the loss function, without needing gradient information or an explicit loss model. However, its classical formulation assumes that the loss evaluation is noise-free, which limits its direct application to practical problems with noisy measurements. We extend blind random search by combining a vanishing acceptance threshold with the running average of noisy loss measurements. Under the stated assumptions, we prove that the proposed algorithm converges almost surely to the global minimizer. We further analyze the rate at which the loss at the current estimate approaches the minimum loss. For low-dimensional cases, we obtain an almost sure asymptotic upper bound on the difference between the two, and for high-dimensional cases, we prove that this difference, suitably scaled, has an exact Weibull limiting distribution.
\end{abstract}

\section{Introduction}

In many practical optimization problems, direct gradient measurements are not
readily available, and measurements of the loss function are often noisy.
Obtaining gradients from system outputs can also require detailed knowledge of
the relationships between system inputs and outputs, whereas methods based on
function measurements can avoid such a detailed model
\cite[pp.~150--153]{spall2003}, \cite{spall1994}.
Random search is the simplest formal optimization method \cite[Chap.~2]{spall2003}.
Random search methods can be competitive when
function characteristics other than function values are difficult to compute
\cite[p.~19]{solis1981}. Among methods based on function measurements,
this paper focuses on blind random search. Candidate points are sampled independently, without using past search
results to guide sampling \cite[pp.~37--38]{spall2003}. In its basic form with uniform sampling, blind random search requires no tuning
of algorithmic coefficients, a property unique among the general stochastic
search and optimization algorithms discussed in \cite[p.~38]{spall2003}. However, it may require a very large number of
loss evaluations to approach the minimizer, especially in high dimensions
\cite[pp.~37, 49--50]{spall2003}.
This paper develops new results for blind random search in the case of noisy
function measurements.

In the absence of noise, blind random search accepts a candidate only when its
loss is below that of the current estimate, so the loss at the current estimate
is nonincreasing \cite[p.~38]{spall2003}. However, loss measurements are often noisy in
practice \cite{spall1994}, so a candidate can appear better while having a
higher true loss \cite{alexander2005}. Random search methods have been used,
for example, for maximum power point tracking in partially shaded photovoltaic
systems \cite{sundareswaran2014}, as well as in broader black-box
optimization procedures \cite{regis2022}. Blind random search, in particular, is
used in robust controller synthesis \cite{vidyasagar2001} and hyperparameter
tuning in machine learning \cite{bergstra2012}. In both cases, candidate points
are sampled independently, and only noisy loss measurements are available.
Averaging and thresholding are two
common approaches for coping with noise in comparison-based search
\cite[pp.~50--51]{spall2003}. Averaging reduces the variance of loss estimates, and a
positive threshold reduces false acceptances caused by noise. Convergence has
been established for several sequential comparison methods that use one or
both approaches \cite{devroye1976,alexander2005},
\cite[Thms.~5--6]{devroye2002}.

In these sequential methods, the loss at the current estimate is measured
anew at every iteration, even when no candidate is accepted and the
algorithm remains at the same $\theta$. Devroye's comparisons use two
independent measurement sequences \cite[p.~53]{devroye1976}. Conditional
on the comparison points, the new measurement noise is independent of
past acceptance and rejection decisions. Such resampling, however,
increases measurement cost. Devroye regarded discarding data after a
comparison as wasteful \cite[p.~706]{devroye1976b} and studied methods that
store earlier search points and their measurements
\cite[Sec.~10]{devroye2002}, \cite{devroye1976b}. These methods also reuse previous averages in later comparisons. They
still take loss measurements at the same points, with the number of
measurements chosen by the user \cite[Sec.~10]{devroye2002},
\cite[Sec.~III]{devroye1976b}. Their algorithms do not use a positive
acceptance threshold \cite[Thm.~7]{devroye2002},
\cite[eq.~(10)]{devroye1976b}, so false acceptances caused by noise are
still possible.

Our algorithm instead keeps the running average at the current estimate
and updates it with exactly one additional loss measurement after each
rejection. It couples this update with a vanishing acceptance threshold. We establish its almost sure convergence and derive the
corresponding convergence rates. Convergence rates also provide a
theoretical basis for comparing the efficiency of stochastic optimization
algorithms \cite{spallhillstark2006}. Rates
for random search have been derived with exact function evaluations by
D\'orea~\cite{dorea1986}, under conditions on the probability of sampling
near the minimizer. For noisy loss measurements, Chia and Glynn~\cite{chia2013}
analyze how a measurement budget is divided between candidate points and
repeated measurements at each point. We analyze the convergence rates
of the proposed algorithm in two regimes, where the regimes are determined
by the dimension of the parameter vector being optimized, obtaining an
almost sure upper rate in one and an exact Weibull limit in the other.
The boundary between the high-dimensional and low-dimensional regimes is
relative rather than absolute: it depends on how steeply the loss function
increases near the minimizer and on the sampling distribution.

Section~II presents the algorithm. Section~III introduces the assumptions,
convergence guarantees, and rate results. Section~IV reports two numerical
experiments, a comparison of the three algorithms, and photovoltaic
simulation results. Section~V is the conclusion.

\section{The Algorithm}
Consider the problem of minimizing a loss function $L(\btheta)$ over a domain
$\Theta\subseteq\R^p$, where only noisy loss measurements
$y(\btheta)=L(\btheta)+\epsilon$ are available. The following modification of
blind random search \cite{spall2003} combines running averages over an
increasing number of loss measurements with a vanishing threshold.

\begin{enumerate}
\item \textbf{Initialization:} Fix an integer $k_0\ge1$ and select an initial
estimate $\thh_0\in\Theta$. Compute the average of $k_0$ loss measurements,
$\yb(\thh_0)=\frac{1}{k_0}\sum_{i=1}^{k_0}y_i(\thh_0)$, and set $k=k_0$ with
$\thh_{k_0}:=\thh_0$ and $\yb(\thh_{k_0}):=\yb(\thh_0)$.
\Needspace{4\baselineskip}
\item \textbf{Search step (iteration $k$):}
\begin{itemize}
\item Randomly generate a candidate $\newth{k+1}\in\Theta$. From $k+1$ new
loss measurements, compute the average at the candidate,
\[
 \yb(\newth{k+1})=\frac{1}{k+1}\sum_{i=1}^{k+1}y_i(\newth{k+1}).
\]
\item Accept the candidate if
\begin{equation}\label{eq:accept}
 \yb(\newth{k+1})<\yb(\thh_k)-\tau_k,
\end{equation}
where the vanishing threshold is
\[
 \tau_k:=c\sigma\sqrt{\frac{\log(k+1)}{k}},\qquad c>2\sqrt2,
\]
and $\sigma$ is the known sub-Gaussian scale parameter formalized in Assumption A2 below.
\item \emph{If accepted:} set $\thh_{k+1}=\newth{k+1}$ and
$\yb(\thh_{k+1})=\yb(\newth{k+1})$.
\item \emph{If rejected:} set $\thh_{k+1}=\thh_k$. Obtain one additional loss
measurement $y_{k+1}(\thh_k)$ and update the average,
\[
 \yb(\thh_{k+1})=\frac{k\,\yb(\thh_k)+y_{k+1}(\thh_k)}{k+1}.
\]
\item At the end of iteration $k$, $\yb(\thh_{k+1})$ is the average of exactly
$k+1$ loss measurements at $\thh_{k+1}$. Increase $k$ by one and return to
the search step.
\end{itemize}
\item \textbf{Termination:} Stop when a predefined criterion is met, such as a
maximum number of noisy loss measurements, or when the user is otherwise
satisfied with the current estimate.
\end{enumerate}

For $k_0=1$, the output after $K$ iterations is $\thh_{K+1}$, obtained using
$N_K$ loss measurements:
\begin{equation}\label{eq:budget}
 N_K=1+\sum_{k=1}^{K}(k+1)+R_K
     =1+\frac{K(K+3)}{2}+R_K.
\end{equation}
Here $R_K$ counts rejected candidates, so $N_K$ grows quadratically with $K$.
Rates below are stated in terms of the iteration index $k$, while the
numerical comparisons use matched measurement budgets.

\section{Almost Sure Convergence and Rates}
Let $\btheta^*$ be a global minimizer of $L$.
The suboptimality gap is $L(\thh_k)-\Lstar$, and the distance to the
minimizer is $\|\thh_k-\btheta^*\|$ (Euclidean norm $\|\cdot\|$).
For $\eta>0$, let $S_\eta:=\{\btheta\in\Theta:L(\btheta)<\Lstar+\eta\}$.
Assumptions A1--A3 are from \cite{zhoujiang2026}; A4 is added for the
rate analysis.
\begin{assumptions}
\item \textbf{A1} \emph{(Separation of the minimizer).} This condition ensures that if the loss $L(\thh_k)$ converges to its minimum $\Lstar$, then the estimate $\thh_k$ converges to the minimizer $\btheta^*$. For every $r>0$,
\begin{equation}\label{eq:separation}
 \inf_{\btheta\in\Theta:\,\|\btheta-\btheta^*\|\ge r}L(\btheta)>\Lstar.
\end{equation}
In particular, the minimizer is unique.
\item \textbf{A2} \emph{(Noise).} This condition specifies the distribution of the measurement noise. All measurement noise terms drawn by the algorithm are i.i.d., centered, and sub-Gaussian with variance proxy $\sigma^2$, as defined by the following inequality:
\begin{equation}\label{eq:subgaussian}
 \E[\exp(\lambda\epsilon)]\le\exp(\lambda^2\sigma^2/2),\qquad \lambda\in\R.
\end{equation}
Examples include Gaussian noise with mean zero, Bernoulli noise with its
mean subtracted, and scaled Rademacher noise, as well as other bounded
noise with mean zero \cite[Sec.~2.1]{wainwright2019}.
\item \textbf{A3} \emph{(Reachability).} There is always a positive probability of sampling a point in the region of nearly minimal loss. Candidates are independent across iterations and independent of the measurement noise. Their sampling distributions are fixed in advance. For every $\eta>0$, there is a $\delta(\eta)>0$, independent of $k$, such that
\begin{equation}\label{eq:reach}
 \Prob\bigl(L(\newth{k+1})<\Lstar+\eta\bigr)\ge\delta(\eta).
\end{equation}
\item \textbf{A4.} There exist $\gamma,\mu,a,\eta_0>0$ such that:

\emph{(i) Local growth.} The loss increases at least as fast as the $\gamma$th power of the distance to $\btheta^*$: for every $\btheta\in\Theta$,
\begin{equation}\label{eq:growth}
 \begin{aligned}
 0\le L(\btheta)-\Lstar<\eta_0\quad\Longrightarrow\quad\\[-2pt]
 L(\btheta)-\Lstar\ge\mu\|\btheta-\btheta^*\|^\gamma.
 \end{aligned}
\end{equation}
For example, $L(\btheta)=\Lstar+\|\btheta-\btheta^*\|^2$ satisfies
(i) with $\gamma=2$, $\mu=1$, and any $\eta_0>0$.

\emph{(ii) Sampling distribution} \cite[p.~78]{dorea1986}.
We strengthen A3 by specifying how fast the probability of sampling a point
in the region of nearly minimal loss decreases as the region shrinks: the
candidate points have a common sampling distribution satisfying
\begin{equation}\label{eq:lower-tail}
 \begin{aligned}
 &\Prob\bigl(L(\newth{j+1})-\Lstar<\eta\bigr)\\[-3pt]
 &\qquad=a\,\eta^{p/\gamma}+o\bigl(\eta^{p/\gamma}\bigr),
 \qquad\eta\to0.
 \end{aligned}
\end{equation}
The constant $a$ is the leading coefficient of the near-optimal sampling
probability, determined by both the loss geometry and the sampling distribution.
For the quadratic example above, uniform sampling on the unit ball centered
at $\btheta^*$ gives $a=1$.

\end{assumptions}

The $\gamma$ here must satisfy both (i) and (ii). For a given loss function
and sampling distribution, (ii) determines $\gamma$ uniquely. The larger the $\gamma$ value, the more gradually the loss decreases as
$\btheta$ approaches $\btheta^*$, making it easier to sample points where the
loss is close to its minimum. In high dimensions, the loss at the current
estimate converges to the minimum more quickly. In low dimensions, the
guaranteed convergence rate is determined by the threshold, so changes in
$\gamma$ do not affect this rate.

Compared with heuristic methods lacking convergence guarantees, the proposed
algorithm is theoretically more reliable. The following theorem establishes
almost sure convergence under A1--A3, retaining the same strength of convergence
as noise-free blind random search \cite[Thm.~2.1]{spall2003}.

\Needspace{8\baselineskip}
\medskip
\begin{theorem}[Almost sure convergence {\cite{zhoujiang2026}}]\label{thm:as}
\mbox{}\par\noindent
Suppose A1--A3 hold, and let $\{\thh_k\}$ be generated by the
algorithm in Section~II. Then, as $k\to\infty$,
\begin{equation}\label{eq:as-result}
 \begin{gathered}
 L(\thh_k)\to\Lstar,\quad \yb(\thh_k)\to\Lstar,\\
 \thh_k\to\btheta^*,
 \end{gathered}
\end{equation}
all almost surely.
\end{theorem}

\begin{proof}
A condensed form of the proof is given here. The complete proof is given
in \cite[Sec.~3]{zhoujiang2026}.
\proofpart{Part I: Comparison noise and selection bias}
Let $\epsb_k^{\mathrm{cur}}$ and $\epsb_k^{\mathrm{new}}$ be the noise
averages at the current estimate and the candidate.
\[
 \begin{aligned}
  \epsb_k^{\mathrm{cur}}&:=\yb(\thh_k)-L(\thh_k),\\
  \epsb_k^{\mathrm{new}}&:=\yb(\newth{k+1})-L(\newth{k+1}).
 \end{aligned}
\]
Set
$D_k:=\epsb_k^{\mathrm{new}}-\epsb_k^{\mathrm{cur}}$ and
$\Delta_k:=L(\newth{k+1})-L(\thh_k)$.
The acceptance test is
\[
 \Delta_k+D_k<-\tau_k.
\]
At a fixed point, a smaller noise $\epsilon$ yields a smaller observed
loss $y=L+\epsilon$. The acceptance rule therefore favors candidates with
smaller average measurement noise, creating selection bias: the average
noise at the current estimate includes the smaller noise values favored by
acceptance, so the average noise at the current estimate generally does not
follow the distribution of an average of the original i.i.d.\ noise terms
assumed in A2 \cite[Sec.~3, Part~I]{zhoujiang2026}.

\proofpart{Part II: The threshold controls comparison noise}
The uniform bound in \cite[Sec.~3, Parts~III--IV]{zhoujiang2026}
accounts for this selection by controlling, before the algorithm's
choices are applied, every sample that could be kept at the current estimate. It gives
\begin{equation}\label{eq:uniform-noise}
 \Prob\bigl(|D_k|>\tau_k\bigr)\le2k(k+1)^{-c^2/4}.
\end{equation}
Since $c^2/4>2$, this bound is summable. The first Borel--Cantelli
lemma therefore gives an a.s.\ finite $J_0$ such that
\begin{equation}\label{eq:eventual-noise}
 |D_k|\le\tau_k\quad\text{for all }k\ge J_0,\quad\text{a.s.}
\end{equation}
In particular, $D_k\to0$ a.s.\ because $\tau_k\to0$.

\proofpart{Part III: Eventual monotonicity}
For $k\ge J_0$, acceptance implies
$\Delta_k<-\tau_k-D_k\le0$; rejection leaves the point unchanged.
Thus $L(\thh_k)$ is eventually nonincreasing a.s.

\proofpart{Part IV: Recurrence near the minimizer}
Fix $\eta>0$. By A3, the independent events
$G_k=\{\newth{k+1}\in S_{\eta/2}\}$ have probabilities at least
$\delta(\eta/2)>0$. Hence $\sum_k\Prob(G_k)=\infty$, and the
second Borel--Cantelli lemma gives $G_k$ infinitely often a.s.

\proofpart{Part V: Eventual acceptance}
For all sufficiently large $k$, $|D_k|\le\tau_k<\eta/4$.
If $\thh_k\notin S_\eta$ and $G_k$ occurs, then $\Delta_k<-\eta/2$, so
\[
 \Delta_k+D_k<-\eta/2+\eta/4=-\eta/4<-\tau_k.
\]
The candidate is accepted into $S_{\eta/2}\subset S_\eta$.
Part~IV ensures such an opportunity eventually occurs, and Part~III
keeps every subsequent iterate in $S_\eta$.

\proofpart{Part VI: Convergence}
Applying Part~V to the countable sequence $\eta=1/m$,
$m\in\mathbb N$, gives $L(\thh_k)\to\Lstar$ almost surely.
A1 then gives $\thh_k\to\btheta^*$ a.s.
For the fresh candidate average, the sub-Gaussian bound
\cite{wainwright2019} gives
$\Prob\bigl(|\epsb_k^{\mathrm{new}}|>\tau_k\bigr)\le2(k+1)^{-c^2/2}$.
Borel--Cantelli and $\tau_k\to0$ imply
$\epsb_k^{\mathrm{new}}\to0$ a.s. Together with Part~II, this yields
$\epsb_k^{\mathrm{cur}}=\epsb_k^{\mathrm{new}}-D_k\to0$ a.s.,
and hence $\yb(\thh_k)\to\Lstar$ a.s.
\end{proof}

Convergence rates are important in practice because they provide a basis for
comparing algorithms and help engineers select algorithms for real-life problems.

\par\medskip
\Needspace{6\baselineskip}
\begin{theorem}[Convergence rates]\label{thm:rates}
\mbox{}\par\noindent
Suppose A1--A4 hold, $\sigma>0$, and
$L(\thh_{k_0})>\Lstar$ almost surely.
For the algorithm in Section~II, the following rates hold as $k\to\infty$.

\par\smallskip
An overestimate or underestimate of $\gamma$ violates A4 and can lead to
an invalid application of Theorem~\ref{thm:rates}. The algorithm itself
does not use $\gamma$, and its almost sure convergence in
Theorem~\ref{thm:as} requires only A1--A3.

\noindent\emph{(i) Suboptimality gap.} If $p>2\gamma$, then
\begin{equation}\label{eq:weibull}
 \begin{aligned}
 (ak)^{\gamma/p}[L(\thh_k)-\Lstar]&\xrightarrow{\mathrm d}W,\\
 \Prob\bigl(W>t\bigr)&=e^{-t^{p/\gamma}},\quad t\ge0.
 \end{aligned}
\end{equation}
Here $W$ is Weibull with shape $p/\gamma$ and unit scale.
If $p\le2\gamma$, then
\begin{equation}\label{eq:upper-rate}
 \limsup_{k\to\infty}
 \frac{L(\thh_k)-\Lstar}{\sqrt{\log k/k}}
 \le2c\sigma\quad\text{a.s.}
\end{equation}
$c$ here is the constant used in the algorithm in Section~II; it is chosen
before the algorithm starts and can be any fixed value satisfying
$c>2\sqrt2$.

\noindent\emph{(ii) Distance to the minimizer.} If $p>2\gamma$, then for every $t>0$,
\begin{equation}\label{eq:parameter-tail-search}
 \limsup_{k\to\infty}\Prob\bigl((ak)^{1/p}\|\thh_k-\btheta^*\|>t\bigr)
 \le e^{-\mu^{p/\gamma}t^p}.
\end{equation}
If $p\le2\gamma$, then
\begin{equation}\label{eq:parameter-tail-noise}
 \limsup_{k\to\infty}
 \frac{\|\thh_k-\btheta^*\|}{(\log k/k)^{1/(2\gamma)}}
 \le\left(\frac{2c\sigma}{\mu}\right)^{1/\gamma}\quad\text{a.s.}
\end{equation}
\end{theorem}

\par\smallskip
D\'orea \cite[p.~78]{dorea1986} obtained a Weibull limit with exact
function evaluations. For smooth loss functions and noisy measurements,
Chia and Glynn \cite[Thm.~3.3(i)]{chia2013} obtained a Weibull limit
when the averaging error is asymptotically negligible relative to the
random-search error. Their nonadaptive algorithm takes the same number
of measurements at each candidate point and reports the best sample
average, without sequential comparisons.

\par\smallskip
\begin{proof}
Let $m=m(k)$ be deterministic, with $k_0\le m<k$ and
$m\to\infty$ as $k\to\infty$. Since the index $J_0$ in
\eqref{eq:eventual-noise} is finite almost surely, $m\ge J_0$
eventually almost surely. For these sufficiently large $k$,
\begin{equation}\label{eq:interval}
 \begin{aligned}
 &L(\thh_k)-\Lstar\\
 &\quad\ge\min\left\{
 \begin{aligned}
 &L(\thh_{k_0})-\Lstar,\\
 &\min_{k_0\le j<k}[L(\newth{j+1})-\Lstar]
 \end{aligned}\right\},\\
 &L(\thh_k)-\Lstar\\
 &\quad\le\min_{m\le j<k}[L(\newth{j+1})-\Lstar]+2\tau_m.
 \end{aligned}
\end{equation}
The lower bound follows because the current estimate is the initial point
or an earlier candidate. For the upper bound, acceptance at $j\ge m$
and subsequent monotonicity give
\[
 L(\thh_k)-\Lstar\le L(\newth{j+1})-\Lstar.
\]
Rejection instead gives
\[
 L(\thh_j)-L(\newth{j+1})\le\tau_j+D_j\le2\tau_j,
\]
and $L(\thh_k)\le L(\thh_j)$.
Taking the minimum and using decreasing $\tau_j$ proves the bound.

Since the candidate suboptimality gaps $L(\newth{j+1})-\Lstar$ are
i.i.d., for every deterministic interval,
\begin{equation}\label{eq:minimum}
 \begin{aligned}
 &\Prob\!\left(\min_{m\le j<k}[L(\newth{j+1})-\Lstar]\ge x\right)\\
 &\quad=\bigl[1-\Prob\bigl(L(\newth{k_0+1})-\Lstar<x\bigr)\bigr]^{k-m}.
 \end{aligned}
\end{equation}
For $p>2\gamma$, take $m=\lfloor k/\log k\rfloor$. Each minimum in
\eqref{eq:interval} is over $n_k$ consecutive candidates, where
$n_k=k-k_0$ or $n_k=k-m$, so $n_k/k\to1$. By A4(ii), for $t>0$,
\[
 \begin{aligned}
 &n_k\,\Prob\!\left(
 L(\newth{k_0+1})-\Lstar<\frac{t}{(ak)^{\gamma/p}}\right)\\
 &\qquad\longrightarrow t^{p/\gamma}.
 \end{aligned}
\]
Hence, by \eqref{eq:minimum}, both minima multiplied by
$(ak)^{\gamma/p}$ have limiting tail probability $e^{-t^{p/\gamma}}$.
Including the initial loss in the lower bound changes this tail
probability by at most
$\Prob\bigl(L(\thh_{k_0})-\Lstar<t/(ak)^{\gamma/p}\bigr)\to0$.
For sufficiently large $k$,
\[
 (ak)^{\gamma/p}\tau_m
 \le\sqrt2c\sigma a^{\gamma/p}
 k^{\gamma/p-1/2}\log k\longrightarrow0
\]
because $\gamma/p<1/2$. The bounds in \eqref{eq:interval} hold whenever
$m\ge J_0$, and $\Prob(J_0>m)\to0$. Their limiting distributions
therefore imply \eqref{eq:weibull}.

For $p\le2\gamma$, set
\[
 k_n=\left\lceil e^{\sqrt n}\right\rceil,\qquad
 m_n=k_n-\left\lceil\frac{k_n}{\sqrt{\log k_n}}\right\rceil.
\]
Then $m_n\to\infty$, $m_n/k_n\to1$, $k_{n+1}/k_n\to1$, and
$k_n-m_n\ge k_n/\sqrt{\log k_n}$. Since $p/(2\gamma)\le1$,
$(\log k_n/k_n)^{p/(2\gamma)}\ge\log k_n/k_n$ for sufficiently large $n$.
Thus A4(ii) and \eqref{eq:minimum} give, for every fixed $u>0$
and all sufficiently large $n$,
\[
 \begin{aligned}
 &\Prob\!\left(
 \frac{\min_{m_n\le j<k_n}[L(\newth{j+1})-\Lstar]}
 {\sqrt{\log k_n/k_n}}\ge u\right)\\
 &\quad\le\exp\!\left[-\frac a2u^{p/\gamma}\sqrt{\log k_n}\right]\\
 &\quad\le\exp\!\left[-\frac a2u^{p/\gamma}n^{1/4}\right].
 \end{aligned}
\]
The last bounds are summable in $n$. The first Borel--Cantelli lemma,
applied to $u=1,1/2,1/3,\ldots$, therefore shows that these interval
minima divided by $\sqrt{\log k_n/k_n}$ tend to zero a.s.
For $k_n\le k<k_{n+1}$ and all sufficiently large $n$,
equation~\eqref{eq:interval} with $m=m_n$, the inclusion
$[m_n,k_n)\subseteq[m_n,k)$, and decreasing $\sqrt{\log k/k}$ give
\[
 \begin{aligned}
 &\frac{L(\thh_k)-\Lstar}{\sqrt{\log k/k}}\\
 &\quad\le
 \frac{\min_{m_n\le j<k_n}[L(\newth{j+1})-\Lstar]+2\tau_{m_n}}
 {\sqrt{\log k_{n+1}/k_{n+1}}}.
 \end{aligned}
\]
Since $m_n/k_n\to1$ and $k_{n+1}/k_n\to1$, the candidate-minimum
term on the right tends to zero a.s.\ and the threshold term tends to
$2c\sigma$. This proves \eqref{eq:upper-rate}.

By Theorem~\ref{thm:as}, $0\le L(\thh_k)-\Lstar<\eta_0$ for all
sufficiently large $k$ almost surely. Thus A4(i) applies to $\thh_k$.
For every $r>0$,
\[
 \begin{aligned}
 \Prob\bigl(\|\thh_k-\btheta^*\|>r\bigr)
 &\le\Prob\bigl(L(\thh_k)-\Lstar\ge\eta_0\bigr)\\
 &\quad+\Prob\bigl(L(\thh_k)-\Lstar>\mu r^\gamma\bigr).
 \end{aligned}
\]
The first term tends to zero. For $p>2\gamma$, substituting
$r=t/(ak)^{1/p}$ into the second term and using \eqref{eq:weibull}
proves \eqref{eq:parameter-tail-search}.
For $p\le2\gamma$, local growth gives, eventually a.s.,
\[
 \frac{\|\thh_k-\btheta^*\|}{(\log k/k)^{1/(2\gamma)}}
 \le\left(\frac{L(\thh_k)-\Lstar}{\mu\sqrt{\log k/k}}\right)^{1/\gamma}.
\]
Taking the limsup and applying \eqref{eq:upper-rate} proves
\eqref{eq:parameter-tail-noise}.
\end{proof}

\section{Numerical Experiments}\label{sec:numerics}
We first compare running averages with resampling and naive search, and then
check the rate results on a photovoltaic model.

\subsection{Convergence behavior and comparison with resampling}
We compare three methods on two problems adapted from
\cite[Examples~2.1 and~6.6]{spall2003}: a quadratic loss
\mbox{$L(\btheta)=\btheta^\top\btheta$} with $p=2$ and $\sigma=2$, and a skewed
quartic loss
\mbox{$L(\btheta)=\sum_{i=1}^{p}[(B\btheta)_i^2
 +0.1(B\btheta)_i^3+0.01(B\btheta)_i^4]$},
where $B$ is the upper triangular matrix of ones, with $p=5$ and
$\sigma=3$; details are given in \cite[Sec.~4]{zhoujiang2026}.
Besides the proposed algorithm with
$c=3$, we run a resampling baseline, which uses the same candidate averages
and threshold but measures the current estimate anew $k$ times at every
iteration, and naive blind random search, which is the standard blind random
search algorithm \cite[p.~38]{spall2003} applied directly to noisy measurements,
treating them as exact loss values. Because loss measurements are usually the
dominant cost \cite[p.~13]{spall2003}, all methods are compared at the same total
number of measurements, over 200 independent replications.

\begin{figure}[t]
\centering
\includegraphics[width=\columnwidth]{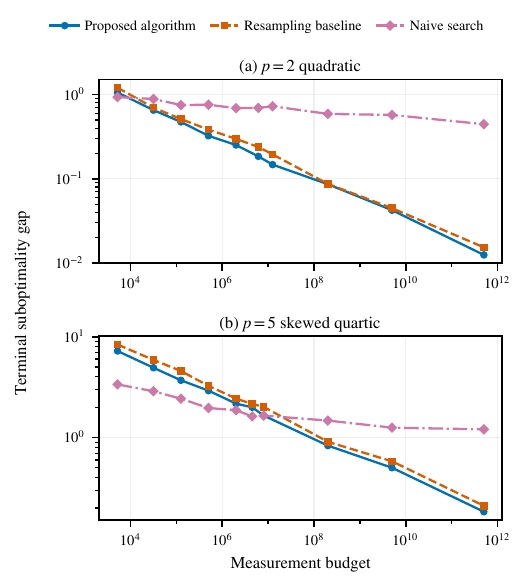}
\caption{Median terminal suboptimality gap versus total measurement budget,
over 200 runs.}
\label{fig:budget-comparison}
\end{figure}

Fig.~\ref{fig:budget-comparison} shows that both averaging methods keep
improving as the budget grows, whereas naive search improves only slowly
after an early lead on the quartic. At the largest budget, the median suboptimality gaps
are $0.0125$ (proposed), $0.0154$ (resampling), and $0.446$ (naive) for the
quadratic, and $0.184$, $0.210$, and $1.20$ for the quartic. Running averages
do better than resampling mainly because they allow $\sqrt2$ times as
many iterations within the same budget. Across the 400 runs at the largest
budget, none of the 2801 accepted updates of the proposed algorithm increased
the true loss.

\begin{figure*}[t]
\centering
\includegraphics[width=\textwidth]{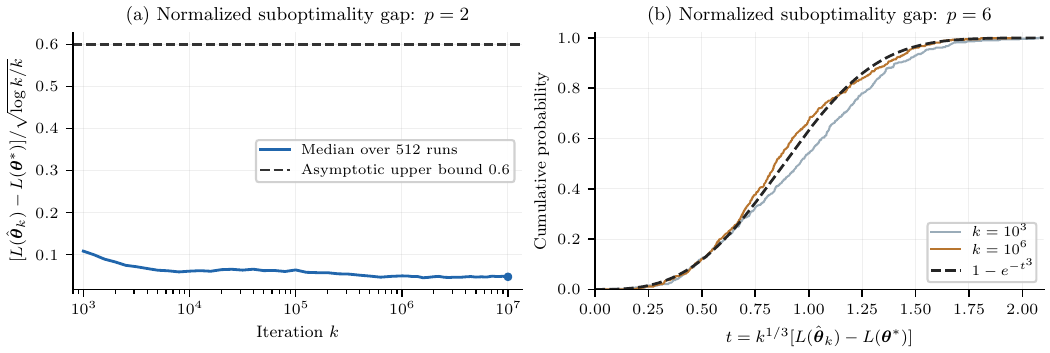}
\caption{Rate diagnostics from 512 runs per dimension. (a) $p=2$: median
of the normalized suboptimality gap and the asymptotic upper bound $0.6$
for the photovoltaic model in Section~\ref{subsec:pv-example}. The bound is
conservative. (b) $p=6$: empirical distributions
of $k^{1/3}[L(\thh_k)-\Lstar]$ at $k=10^3$ and $10^6$, and the Weibull limit
$1-e^{-t^3}$.}
\label{fig:pv-rates}
\end{figure*}

\subsection{Photovoltaic example and convergence rates}\label{subsec:pv-example}
We apply Theorems~\ref{thm:as} and~\ref{thm:rates} to a simplified
photovoltaic model. Partial shading can produce competing power peaks,
including in systems with several converter duty ratios \cite{chao2016},
and quadratic interpolation of power curves is a related approximation
\cite{liu2014}.

\par\smallskip
\Needspace{4\baselineskip}
\noindent\emph{Model and objective.}
The array operates under fixed shading and temperature. The control vector
$\btheta\in\Theta$, where $\Theta=\{\btheta\in\R^p:\|\btheta\|\le1\}$,
sets the duty ratios $0.50+0.20\theta_i$. Let
$\btheta^*=(2\sqrt p)^{-1}(1,\ldots,1)^\top$, so that
$\|\btheta^*\|=1/2$. Define the loss
\begin{equation}\label{eq:pv-gap}
 L(\btheta)=\min\{\|\btheta-\btheta^*\|^2,\;1/4+\|\btheta+\btheta^*\|^2\},
\end{equation}
and the total power, measured in watts, by
\[
 P(\btheta)=100-20L(\btheta).
\]
Minimizing the loss is equivalent to maximizing the total power. The
global peak at $\btheta^*$ is $100$~W with $\Lstar=0$; the inferior
local peak at $-\btheta^*$ is $95$~W.

\par\smallskip
\noindent\emph{Measurements and sampling.}
The noisy loss measurements are
\[
 y(\btheta)=\min\bigl\{\|\btheta-\btheta^*\|^2,\;
 \tfrac14+\|\btheta+\btheta^*\|^2\bigr\}+\epsilon,
\]
where the noise terms are i.i.d.\ $N(0,\sigma^2)$ with \mbox{$\sigma=0.1$}.
Candidate points are sampled independently and uniformly from $\Theta$.
\par\medskip
\Needspace{9\baselineskip}
\noindent\emph{Application model and its convergence rate.}\par
\noindent
For the model above, let $k_0=1$, $\thh_1=-\btheta^*$, and
$c=3>2\sqrt2$. Then A1--A4 hold with $\gamma=2$, $\mu=a=1$, and
$\eta_0=1/4$, and $\thh_k\to\btheta^*$ almost surely. For $p\le4$,
\begin{align}
 \limsup_{k\to\infty}
 \frac{L(\thh_k)-\Lstar}{\sqrt{\log k/k}}
 &\le0.6\quad\text{a.s.},\label{eq:pv-as-loss}\\
 \limsup_{k\to\infty}
 \frac{\|\thh_k-\btheta^*\|}{(\log k/k)^{1/4}}
 &\le\sqrt{0.6}\quad\text{a.s.}\label{eq:pv-as-parameter}
\end{align}
For $p>4$ and $t\ge0$,
\begin{equation}\label{eq:pv-weibull}
 \begin{aligned}
 \Prob\bigl(k^{2/p}[L(\thh_k)-\Lstar]>t\bigr)
 &\longrightarrow e^{-t^{p/2}},\\
 \Prob\bigl(k^{1/p}\|\thh_k-\btheta^*\|>t\bigr)
 &\longrightarrow e^{-t^{p}}.
 \end{aligned}
\end{equation}
\par\smallskip
\noindent\emph{Assumption checks.}
A1 follows from \eqref{eq:pv-gap}: $L(\btheta)\ge\min(r^2,1/4)>0$
whenever $\|\btheta-\btheta^*\|\ge r>0$. A2 and the independence
requirements of A3 follow from the specified Gaussian noise and sampling scheme.

For $0<\eta<1/4$, $\{L<\eta\}$ is the ball of radius $\sqrt{\eta}$
centered at $\btheta^*$, contained in $\Theta$ because
$\|\btheta^*\|+\sqrt{\eta}<1$. Its sampling probability $\eta^{p/2}$
and the identity $L(\btheta)=\|\btheta-\btheta^*\|^2$ on this ball
give A3 and A4 with $\gamma=2$, $\mu=a=1$, and $\eta_0=1/4$.

Since $L(\thh_1)=1/4>\Lstar$, Theorem~\ref{thm:rates} applies,
and by Theorem~\ref{thm:as}, $L(\thh_k)<1/4$ for all large $k$
almost surely, hence $L(\thh_k)=\|\thh_k-\btheta^*\|^2$ for all large $k$.
For $p\le2\gamma=4$, \eqref{eq:upper-rate} and
\eqref{eq:parameter-tail-noise} with $2c\sigma=2\times3\times0.1=0.6$
and $\mu=1$ give \eqref{eq:pv-as-loss} and \eqref{eq:pv-as-parameter}.
For $p>4$, substituting $a=1$ and $\gamma=2$ into \eqref{eq:weibull}
gives the first limit in \eqref{eq:pv-weibull}; taking square roots through
the identity above turns $e^{-(t^2)^{p/2}}$ into $e^{-t^p}$, which is
the second.

\par\medskip
\noindent\emph{Simulation setup.}
We use $k_0=1$, $\thh_1=-\btheta^*$, and $c=3$, with 512 independent
replications per dimension. The cases $p=2$ and $p=6$ lie on opposite
sides of $2\gamma=4$:
\par\noindent
For $p=2$, we run to $k=10^7$ and compare the median
of $[L(\thh_k)-\Lstar]/\sqrt{\log k/k}$ against the bound~$0.6$.

\par\noindent
For $p=6$, we run to $k=10^6$ and compare the
empirical CDF of $k^{1/3}[L(\thh_k)-\Lstar]$ with the Weibull CDF
$1-e^{-t^3}$.
\par
The distance to the minimizer is obtained from
$\|\thh_k-\btheta^*\|^2=L(\thh_k)-\Lstar$, which holds in every run for
$k\ge10^3$.

Fig.~\ref{fig:pv-rates}(a) shows that, for $p=2$, the bound is
conservative: the median of the normalized suboptimality gap is $0.108$ at
$k=10^3$ and $0.047$ at $k=10^7$, and no run exceeds the bound at any
reported $k\ge10^3$.
For $p=6$, Fig.~\ref{fig:pv-rates}(b) shows that the empirical
distribution at $k=10^6$ is consistent with the Weibull limit: its median
of $0.84$ is close to the limit's median of $0.885$, and the largest
distance between its empirical CDF and the limit CDF is $0.057$. At
$k=10^3$, also shown in Fig.~\ref{fig:pv-rates}(b), the threshold still
shifts the distribution to the right (median $0.96$). Therefore, for
$p=6$, the suboptimality gap decreases as $k^{-1/3}$, and the distance to
the minimizer as $k^{-1/6}$.

\FloatBarrier
\section{Conclusions}
This paper extends blind random search to noisy loss measurements by
combining running averages and a decreasing acceptance threshold. The
algorithm converges almost surely to the global minimizer.

In high dimensions, the normalized suboptimality gap has a Weibull
limit; in low dimensions, we establish an almost sure asymptotic upper
bound on the suboptimality gap. This bound decreases to zero as the
number of iterations increases. These rates help engineers assess the
asymptotic convergence speed of our algorithm.

In future work, we will explore applying the combination of averaging and
thresholding to other stochastic optimization algorithms, including local
random search and simulated annealing.

\Needspace{8\baselineskip}

\end{document}